\documentclass[10pt]{article}

\usepackage{amssymb, amsmath}
\usepackage{amsthm}
\usepackage{graphics}
\usepackage{amsfonts}
\usepackage{mathrsfs}
\usepackage{amscd}
\usepackage{comment}
\usepackage{color}
\usepackage{url}
\newcommand{\bN} { {\mathbb{N}}}
\newcommand{\bC} { {\mathbb{C}}}
\newcommand{\bQ} { {\mathbb{Q}}}
\newcommand{\bZ} { {\mathbb{Z}}}

\newcommand{\bF} { {\mathbb{F}}}

\newcommand{\bE} { {\mathbb{E}}}
\newcommand{\abs}[1]{\lvert#1\rvert}

\newcommand{\simq}{{\sim_{q^\bZ}}}
\newcommand{\si} { {\sigma}}

\newcommand{\disp}{\operatorname{disp}}
\newcommand{\qdisp}{\operatorname{qdisp}}

\newcommand{\dres}{\operatorname{dres}}
\newcommand{\qres}{\operatorname{qres}}
\newcommand{\Tr}{{\operatorname{Tr}}}

\newcommand{\bark}{{\overline{k}}}

\newcommand{\QED}{{ }}

\newtheorem{theorem}{Theorem}[section]
\newtheorem{lemma}[theorem]{Lemma}

\newtheorem{definition}[theorem]{Definition}
\newtheorem{fact}[theorem]{Fact}
\newtheorem{prop}[theorem]{Proposition}
\newtheorem{example}[theorem]{Example}

\newtheorem{remark}[theorem]{Remark}

\begin{document}

\title{On the Summability of Bivariate Rational Functions\thanks{This work was supported in part by NSF Grant CCF-1017217.}
}
\author{Shaoshi Chen and Michael F.\ Singer\footnote{Email Addresses:~{\tt schen21@ncsu.edu}~(S.\ Chen),
{\tt singer@math.ncsu.edu}~(M.\ F.\ Singer).}\, \footnote{Corresponding Author: Michael F. Singer, Department of Mathematics,
North Carolina State University, Box 8205, Raleigh, NC 27695-8205,
{\tt singer@math.ncsu.edu}, Telephone:  919 515 2671, Fax:  919 515 3798} }


\maketitle

\begin{abstract}
We present criteria for deciding whether a bivariate rational function in two variables
can be written as a sum of two ($q$-)differences of bivariate rational functions. Using these criteria, we show how certain double sums
can be evaluated, first, in terms of single sums and, finally, in terms of values of special functions.
\end{abstract}

\section{Introduction}\label{SECT:intro}
 The method of residues has been a powerful tool for investigating various problems
in algebra, analysis, and
combinatorics~\cite{Gelfand1971, Egorychev1984, Mitrinovic1984, Mitrinovic1993}.
This paper is a further example of the method of residues the authors used in~\cite{ChenSinger2012}.
By focusing on residues, we are able to give a unified approach to problems in the shift and
$q$-shift cases while also identifying where these cases differ.

The general question considered in this paper was raised by Andrews and Paule in~\cite{Andrews1993}:
\begin{center}
\begin{minipage}[t]{11cm}
\lq\lq Is it possible to provide any algorithmic device for
reducing multiple sums to single ones?\rq\rq
\end{minipage}
\end{center}
The single sums are much more easily handled due to the celebrated Gopser algorithm~\cite{Gosper1978}
for hypergeometric terms.
The Gosper algorithm decides whether a hypergeometric term~$T(n)$ is equal to the difference of
another hypergeometric term.
If such a hypergeometric term exists, we say that~$T(n)$ is~\emph{hypergeometric summable}.
Passing from univariate to multivariate, the first step has been started in the work
by Chen et al.~in~\cite{ChenHouMu2006}. However, they only obtained necessary conditions for hypergeometric summability of
bivariate hypergeometric terms. As a starting point, we focus on the double sums of rational functions.
With the help of the discrete and $q$-discrete analogues of usual complex residues in analysis,
we derive necessary and sufficient conditions, which allow us to decide whether a rational
function in two variables can be written as a sum of two ($q$-)differences of rational functions.

For a precise description, let~$\bF$ be an algebraically closed field of characteristic zero and~$\bF(x, y)$ be the field
of rational functions in~$x$ and~$y$ over~$\bF$. Let~$\phi$ and~$\varphi$ be two automorphisms of~$\bF(x, y)$.
A rational function~$f\in \bF(x, y)$ is said to be~\emph{$(\phi, \varphi)$-summable} in~$\bF(x, y)$ if there exist~$g, h\in \bF(x, y)$
such that
\[f = \phi(g) - g  + \varphi(h) -h.\]
The problem we study is the following.

\begin{center}
\begin{minipage}[t]{11.8cm}{\bf Bivariate Summability Problem.}
    Given a rational function~$f\in \bF(x, y)$, decide whether or not~$f$ is $(\phi, \varphi)$-summable in~$\bF(x, y)$.
\end{minipage}
\end{center}
To make the problem more tractable, we will make some restriction on~$\phi$ and~$\varphi$.
We define shift operators~$\si_x$ and~$\si_y$ in~$\bF(x, y)$
as
\[\si_x(f(x, y)) = f(x+1, y) \quad \text{and}\quad \si_y(f(x, y))= f(x, y+1)\]
for all~$f\in \bF(x, y)$.
For~$q\in \bF\setminus \{0\}$, we define~$q$-shift operators~$\tau_x$ and~$\tau_y$ in~$\bF(x, y)$ as
\[\tau_x(f(x, y)) = f(qx, y) \quad \text{and}\quad \tau_y(f(x, y))= f(x, qy) \quad \text{for all~$f\in \bF(x, y)$.}\]
In this paper, we will solve the problem above in two cases: one is that~$\phi=\si_x$ and~$\varphi = \si_y$ and the other is that
$\phi=\tau_x$ and~$\varphi = \tau_y$.

The continuous counterpart, namely \emph{bivariate integrability problem}, \linebreak traces back to the works by Poincar\'{e}~\cite{Poincare1887}
and Picard~\cite{Picard1897}.  Let~$D_x$ and~$D_y$ denote the derivations with respect to~$x$ and~$y$, respectively.
The problem is to decide whether a rational function~$f\in \bF(x, y)$
is equal to~$D_x(g) + D_y(h)$ for some~$g, h\in \bF(x, y)$.
In~\cite[vol 2, page 220]{Picard1897}, Picard gave a necessary and sufficient condition, which says that such a pair~$(g, h)$ exists for~$f$ if and
only if all residues of~$f$ with respect to~$y$ as algebraic functions in~$\overline{\bF(x)}$
are equal to derivatives of other algebraic functions.
For a more elementary proof of Picard's criterion and many applications, one can see the paper~\cite{ChenKauersSinger2012}.
So the criteria in this paper can be viewed as a discrete and $q$-discrete analogue of Picard's criterion.

A gallery of all results can be illustrated by the rational function
\[f = \frac{1}{x^n + y^n}, \quad \text{where~$n\in \bN\setminus \{0\}$.}\]
\begin{itemize}
  \item The continuous case (Example 5 in~\cite{ChenKauersSinger2012}):
  \[ \text{$f = D_x(g) + D_y(h)$ for some~$g, h \in \bF(x, y)$ $\quad \Leftrightarrow \quad$ ~$n\neq 2$.} \]
  \item The discrete case (Example~\ref{EXAM:ratsum} below):
  \[ \text{$f = \si_x(g)-g + \si_y(h)-h$ for some~$g, h \in \bF(x, y)$ $\quad \Leftrightarrow \quad$ ~$n=1$.} \]
  \item The $q$-discrete case:
  \begin{itemize}
    \item $q$ is a root of unity with~$q^m=1$ and $m$ minimal (Example~\ref{EXAM:qratsumru} below):
\[ \text{$f = \tau_x(g)-g + \tau_y(h)-h$ for some~$g, h \in \bF(x, y)  \Leftrightarrow$ $n \neq 0\, {\rm mod} \, m$.} \]
    \item $q$ is not a root of unity (Example~\ref{EXAM:qratsumnru} below): For all $n\in \bN\setminus \{0\}$, 
        \[ \text{ $f = \tau_x(g)-g + \tau_y(h)-h$ for some $g, h \in \bF(x, y)$}. \]
  \end{itemize}
\end{itemize}

Using the ($q$)-summability criteria, we show some identities between double sums and single sums. For instance,
\[\sum_{n=1}^{\infty}\sum_{m=1}^{\infty} \frac{1}{(m+n/2)^3} = 4\zeta(2)-\frac{9}{2}\zeta(3),\]
and
\[\sum_{a= 1}^\infty\sum_{b= 1}^\infty\frac{1}{q^{an} + q^{bn}} = \frac{1}{1-q^n}\left(-\frac{1}{2} + 2 {\rm L}_1(-1,1/q^n))\right)\]
where ${\rm L}_1(x,q)$ is the $q$-logarithm (see Example~\ref{EXAM:qratsumnru}).

The rest of this paper is organized as follows. In Section~\ref{SECT:dres}, we recall the notion of discrete
residues and their  $q$-analogues, introduced in~\cite{ChenSinger2012}. In terms of these residues, we
review the necessary and sufficient conditions for the summability of univariate rational functions.
The importance of discrete residues and their $q$-analogues lies in reducing the problem of summability of bivariate rational
functions to that of summability of univariate algebraic
functions. In Section~\ref{SECT:criteria}, we present a necessary and sufficient condition on the
summability of rational functions in two variables and also some examples.

\section{Summability problem: the univariate case}\label{SECT:dres}
Let~$\bE$ be an algebraically closed field of characteristic zero.
In the next section, we will take~$\bE$ to be the algebraic closure of the field~$\bF(x)$.
Let~$\bE(y)$ be the field of rational functions in~$y$ over~$\bE$.
Let~$\phi$ be an automorphism of~$\bE(y)$ that fixes~$\bE$. A rational function~$f\in \bE(y)$
is said to be~\emph{$\phi$-summable} in~$\bE(y)$ if~$f = \phi(g) -g$ for some~$g\in \bE(y)$.
The goal of this section is to solve the following problem.
\begin{center}
\begin{minipage}[t]{11.8cm}{\bf Univariate Summability Problem.}
For a given~$\bE$-automorphism~$\phi$ of~$\bE(y)$ and~$f\in \bE(y)$,
decide whether~$f$ is~$\phi$-summable in~$\bE(y)$ or not.
\end{minipage}
\end{center}
This problem will be reduced into two special summability problems, which have been
extensively studied in~\cite{Abramov1975, Pirastu1995a, Paule1995b, Marshall2005, Matusevich2000, ChenSinger2012}.
It is well-known that~$\phi$
is the linear fractional transformation (see~\cite[p.\ 181-182]{Weintraub2009}) uniquely determined by
\[\phi(y) = \frac{ay+b}{cy+d}, \quad \text{where~$a, b, c, d\in \bE$ and~$ad-bc\neq 0$.}\]
Let~$A$ denote the matrix~$\bigl(\begin{smallmatrix}
a&b\\ c&d
\end{smallmatrix} \bigr)$. Then the action of~$A$ on~$\bE(y)$ can be naturally defined as
\[A(f(y)) = f\left(\frac{ay+b}{cy+d}\right).\]
Let~$A=BJB^{-1}$ be the Jordan decomposition of~$A$ over~${\bE}$ with~$B\in \text{GL}(2, {\bE})$ and~$J$ is
one of the following forms:
\begin{itemize}
  \item[(i).] The shift case:
  \[ J = \begin{pmatrix}
\lambda & 1\\ 0 & \lambda
\end{pmatrix}, \quad \text{where~$\lambda\in {\bE}\setminus \{0\}$.}\]
In this case, we have~$J(y)  = y + \frac{1}{\lambda}$. Furthermore, we decompose~$J$ into
the product~$\varphi \si \varphi^{-1}$, where~$\varphi(y) = \frac{y}{\lambda}$ and~$\si(y)=y+1$.
  \item[(ii).] The~$q$-shift case:
  \[J = \begin{pmatrix}
\lambda & 0\\ 0&\mu
\end{pmatrix}, \quad \text{where~$\lambda, \mu\in {\bE}\setminus \{0\}$.}\]
In this case, we have~$J(y) = qy$ with~$q=\lambda/\mu \in {\bE}$.
\end{itemize}

As in the introduction, we let~$\si_y, \tau_y$ denote the shift and~$q$-shift operators with respect to~$y$ in~$\bE(y)$, respectively.
Let~$\phi_1, \phi_2$ be two $\bE$-automorphisms of~$\bE(y)$
such that~$\phi_1 = \varphi \phi_2 \varphi^{-1}$ for some $\bE$-automorphisms~$\varphi$ of~$\bE(y)$. Then the problem
of deciding whether~$f\in \bE(y)$ is~$\phi_1$-summable in~$\bE(y)$ or not is equivalent to that of deciding
whether~$\varphi^{-1}(f)$ is~$\phi_2$-summable in~$\bE(y)$ or not.
According to the discussion as above, any~$\bE$-automorphism is similar to either the shift operator
or the~$q$-shift operator over~${\bE}$. So the {\bf Univariate Summability Problem} can be
reduced into the usual summability and~$q$-summability problems. We will discuss those two cases separately.

\subsection{The shift case} \label{SUBSECT:ushift}
In this case, we consider the problem of deciding whether a given rational function~$f\in \bE(y)$
is equal to the difference~$\si_y(g)-g$ for some~$g\in \bE(y)$. We review
a necessary and sufficient condition in terms of the discrete analogue of the usual residues in complex analysis
from~\cite{ChenSinger2012}.

For an element~$\alpha \in {\bE}$, we call the subset~$\alpha + \bZ$ the~\emph{$\bZ$-orbit} of~$\alpha$
in~${\bE}$, denoted by~$[\alpha]$. Two elements~$\alpha_1, \alpha_2$ are said to be~\emph{$\bZ$-equivalent} if
they are in the same~$\bZ$-orbit, denoted by~$\alpha_1\sim_{\bZ}\alpha_2$.
For a polynomial~$p\in \bE[y]\setminus \bE$, the value
\[\max\{i\in \bZ \mid \text{$\exists \, \alpha, \beta \in {\bE}$
such that~$i=\alpha-\beta$ and~$p(\alpha)=p(\beta)=0$}\}\]
is called the~\emph{dispersion} of~$p$ with respect to~$y$,
denoted by~$\disp_y(p)$. A polynomial~$p\in \bE[y]$ is said to
be~\emph{shift-free} with respect to~$y$ if~$\disp_y(p)=0$.
Let~$f=a/b\in \bE(y)$ be such that~$a, b\in \bE[y]$ and~$\gcd(a, b)=1$.
Since the field~${\bE}$ is algebraically closed, $f$~can be decomposed into the form
\begin{equation}\label{EQ:dparfrac}
f =  p + \sum_{i=1}^m \sum_{j=1}^{n_i} \sum_{\ell=0}^{d_{i,j}} \frac{\alpha_{i,j,\ell}}{(\si_y^{-\ell}(y)-\beta_i)^{j}},
\end{equation}
where~$p\in \bE[y]$,  $m, n_i, d_{i,j}\in \bN$, $\alpha_{i, j, \ell}, \beta_i\in {\bE}$, and~$\beta_i$'s are
in distinct $\bZ$-orbits. We introduce a discrete analogue of the usual residues for rational functions, which is
motivated by the following fact.
\begin{fact}\label{FACT:1}
Let~$\phi$ be any~$\bE$-automorphism of~$\bE(y)$ and~$\alpha, \beta\in \bE$. Then for all~$m, n$ in~$\bN$ we have
\[\frac{\alpha}{(\phi^n(y) -\beta)^m} = \phi(g)-g + \frac{\alpha}{(y-\beta)^m},\quad
\text{where~$g=\sum_{j=0}^{n-1}\frac{\alpha}{(\phi^{j}(y)-\beta)^m}$}.\]
\end{fact}

\begin{definition}[Discrete residue]\label{DEF:dres}
Let~$f\in \bE(y)$ be of the form~\eqref{EQ:dparfrac}.
The sum $\sum_{\ell=0}^{d_{i,j}}\alpha_{i,j, \ell}\in {\bE}$
is called the \emph{discrete residue} of~$f$ at
the $\bZ$-orbit $[\beta_i]$ of multiplicity~$j$ with respect to~$y$,
denoted by~$\dres_y(f, [\beta_i], j)$.
\end{definition}

We recall a criterion on the summability in~$\bE(y)$
via discrete residues.
\begin{prop}[c.f. Prop. 2.5 in~\cite{ChenSinger2012}]\label{PROP:ratsum}
Let~$f=a/b\in \bE(y)$ be such that~$a, b\in \bE[y]$ and~$\gcd(a, b)=1$.
Then~$f$ is $\si_y$-summable in~$\bE(y)$ if and only if the discrete residue $\dres_y(f, [\beta], j)$
is zero for any $\bZ$-orbit~$[\beta]$ with~$b(\beta)=0$ of any multiplicity~$j\in \bN$.
\end{prop}
In terms of discrete residues, we derive a normal form for a rational function in the
quotient space~$\bE(y)/((\si_y-1)(\bE(y)))$. Let~$f$ be of the form~\eqref{EQ:dparfrac}.
Then we can decompose it into~$f = \si_y(g)-g + r$, where~$g, r\in \bE(y)$ and
\[\text{$r=\sum_{i=1}^m\sum_{j=1}^{n_i} \frac{\dres_y(f, [\beta_i], j)}{(y-\beta_i)^j}$
\, \, with~$\beta_i$'s being in distinct $\bZ$-orbits.}\]
The condition above on~$r$ is equivalent to the condition that the rational function~$r$ is proper and its denominator is
shift-free with respect to~$y$.
Such an additive decomposition can be computed
by the algorithms in~\cite{Abramov1975, Abramov1995b, Paule1995b, Pirastu1995a, Pirastu1995b}.

\subsection{The $q$-shift case}\label{SUBSECT:uqshift}
Let~$q$ be an element of~$\bE$. We consider the problem of deciding whether a
given rational function~$f\in \bE(y)$ is equal to the difference~$\tau_y(g)-g$
for some~$g\in \bE(y)$.

We first study the case in which~$q$ is a root of unity.
Assume that~$m$ is the minimal positive integer such that~$q^m=1$. We  do not assume that $\bE$ is algebraically closed
but rather only assume that $\bE$ contains all $m^{th}$ roots of unity.
It is easy to show that~$\tau_y(f)=f$ if and only if~$f\in \bE(y^m)$.
Let~$p =z^m -y^m\in \bE(y^m)[z]$.
By the assumption that~$\bE$ contains all $m^{th}$ roots of unity, $\bE(y)$ is the splitting field of~$p$ over~$\bE(y^m)$.
Since~$\bE$ is of characteristic zero, $\bE(y)$ is a Galois extension of~$\bE(y^m)$ and its Galois group is cyclic and generated by~$\tau_y$.
 We now derive a normal form for rational functions in~$\bE(y)$ with respect to~$\tau_y$.
\begin{lemma}\label{LM:qru}
Let~$q$ be such that~$q^m=1$ with~$m$ minimal and let $f \in \bE(y)$.\begin{enumerate}
\item[(a)] $f = \tau_y(g) - g$ for some $g \in \bE(y)$ if and only if the trace $\Tr_{\bE(y)/\bE(y^m)}(f) = 0$.
\item[(b)] Any rational function~$f\in \bE(y)$ can be decomposed into
\begin{equation}\label{rook}
f = \tau_y(g) - g + c, \quad \text{where~$g\in \bE(y)$ and~$c\in \bE(y^m)$}.
\end{equation}
Moreover, $f$ is~$\tau_y$-summable in~$\bE(y)$ if and only if~$c=0$.
\end{enumerate}
\end{lemma}
\begin{proof} (a) This is just a restatement of the additive version of Hilbert's Theorem 90 (see \cite[Thm. 6.3, p.~290]{LANG}).\\[0.1in]
(b) Since~$f$ is algebraic over~$\bE(y^m)$ and~$[\bE(y):\bE(y^m)] = m$,
we can write~$f$ as
\begin{equation*}
f = a_{m-1}y^{m-1} + \cdots + a_0, \quad \text{where~$a_0, \ldots, a_{m-1} \in \bE(y^m)$}.
\end{equation*}
Since $\Tr_{\bE(y)/\bE(y^m)}(y^i) = 0$ for $i = 1, \ldots , m-1$,  the assertion in part (a)
implies that $f = \tau_y(g) - g + a_0$ for some~$g\in \bE(y)$
(alternatively, note that for each~$i\in \{1, \ldots, m-1\}$, we have~$y^i = \tau_y(g_i)-g_i$
with~$g_i = \frac{y^i}{q^i-1}$). So~$f-a_0$ is~$\tau_y$-summable in~$\bE(y)$.
For any nonzero element~$c\in \bE(x^m)$, the trace of~$c$ is not zero.
So~$f$ is~$\tau_y$-summable if and only if~$a_0$ is zero.
\QED
\end{proof}

In this case, we see that it is quite easy to verify the~$q$-summability of rational functions in~$\bE(y)$.\\[0.1in]
Now we assume that~$q$ is not a root of unity and return to the assumption that $\bE$ is algebraically closed.
For an element~$\alpha \in \bE,$, we call the subset~$\{\alpha \cdot q^i \mid i \in \bZ\}$ of~$\bE$
the~\emph{$q^\bZ$-orbit} of~$\alpha$
in~$\bE$, denoted by~$[\alpha]_q$. We say $\alpha$ and $\beta$ are $q^\bZ$-equivalent if $\beta \in [\alpha]_q$ and
we write $\alpha\simq \beta$. For a
polynomial~$b\in \bE[y],  b \neq \lambda y^n, \lambda \in \bE, n \in \bN$, the value
\[\max\{i\in \bZ \mid \text{$\exists$  nonzero~$\alpha, \beta \in \bE$ such
that~$\alpha=q^i\cdot \beta$ and~$b(\alpha)=b(\beta)=0$}\}\]
is called the~\emph{$q$-dispersion} of~$b$ with respect to~$x$, denoted by~$\qdisp_x(b)$.
For~$b=\lambda y^n$ with~$\lambda \in \bE$ and~$n\in \bN\setminus\{0\}$,
we define~$\qdisp_y(b)=+\infty$.
The polynomial~$b$ is said to be~\emph{$q$-shift-free} with respect
to~$y$ if~$\qdisp_y(b)=0$.
Let~$f=a/b\in \bE(y)$ be such that~$a, b\in \bE[y]$ and~$\gcd(a, b)=1$.
Over the field~$\bE$, $f$ can be uniquely decomposed into the form
\begin{equation}\label{EQ:qparfrac}
f = c + yp_1 + \frac{p_2}{y^s} +
\sum_{i=1}^m \sum_{j=1}^{n_i} \sum_{\ell=0}^{d_{i,j}} \frac{\alpha_{i,j,\ell}}{(\tau_y^{-\ell}(y)-\beta_i)^{j}},
\end{equation}
where~$c\in \bE$, $p_1, p_2 \in \bE[y]$,  $m, n_i\in \bN$ are nonzero,
$s, d_{i,j}\in \bN$, $\alpha_{i, j, \ell}, \beta_i\in \bE$, and~$\beta_i$'s are nonzero and
in distinct $q^\bZ$-orbits. Motivated by Fact~\ref{FACT:1}, we introduce a $q$-discrete analogue of the
usual residues for rational functions.

\begin{definition}[$q$-discrete residue]\label{DEF:qres}
Let~$f\in \bE(y)$ be of the form~\eqref{EQ:qparfrac}.
The sum $\sum_{\ell=0}^{d_{i,j}}\alpha_{i,j, \ell}$
is called the \emph{$q$-discrete residue} of~$f$ at
the~$q^\bZ$-orbit $[\beta_i]_q$ of multiplicity~$j$ (with respect to~$y$),
denoted by~$\qres_x(f, [\beta_i]_q, j)$. In addition, we call the constant~$c$ the
\emph{$q$-discrete residue} of~$f$ at infinity, denoted by~$\qres_x(f, \infty)$.
\end{definition}
\begin{remark}
One should notice that the definition of $q$-discrete residues in~\cite{ChenSinger2012} is
defined via the decomposition
\begin{equation}\label{EQ:qparfrac2}
f = c + yp_1 + \frac{p_2}{y^s} +
\sum_{i=1}^m \sum_{j=1}^{n_i} \sum_{\ell=0}^{d_{i,j}} \frac{\bar{\alpha}_{i,j,\ell}}{(y-q^\ell \cdot \beta_i)^{j}},
\end{equation}
and~$\qres_x(f, [\beta_i]_q, j) = \sum_{\ell=0}^{d_{i,j}}q^{-\ell \cdot j} \bar{\alpha}_{i,j, \ell}$. But it is easy to see
that~$\alpha_{i, j, \ell} = q^{-\ell \cdot j} \bar{\alpha}_{i,j, \ell}$. Therefore, the two definitions coincide.
This adjustment will allow us treat discrete residues and their~$q$-analogue in a more similar way.
\end{remark}

The following lemma is a~$q$-analogue of Proposition~\ref{PROP:ratsum}.
\begin{prop}[c.f. Prop. 2.10 in~\cite{ChenSinger2012}]\label{PROP:ratqsum}
Let~$f=a/b\in \bE(y)$ be such that~$a, b\in \bE[y]$ and~$\gcd(a, b)=1$.
Then~$f$ is rational $\tau_y$-summable in~$\bE(y)$ if and only if the~$q$-discrete
residues $\qres_y(f, \infty)$
and~$\qres_y(f, [\beta]_q, j)$ are all zero for any $q^\bZ$-orbit~$[\beta]_q$ with~$\beta \neq 0$ and~$b(\beta)=0$
of any multiplicity~$j\in \bN$.
\end{prop}

In terms of $q$-discrete residues, we derive a normal form for a rational function in the
quotient space~$\bE(y)/((\tau_y-1)(\bE(y)))$. Let~$f$ be of the form~\eqref{EQ:qparfrac}.
Then we can decompose it into~$f = \tau_y(g)-g + r$, where~$g, r\in \bE(y)$ and
\[\text{$r= c + \sum_{i=1}^m\sum_{j=1}^{n_i} \frac{\qres_y(f, [\beta_i]_q, j)}{(y-\beta_i)^j}$
\, \, with~$\beta_i$'s being in distinct $q^{\bZ}$-orbits.}\]
The condition above on~$r$ is equivalent to that the rational function~$r-c$ is proper and its denominator is
$q$-shift-free with respect to~$y$.
Such an additive decomposition can be computed
by the algorithms in~\cite{Abramov1975, Abramov1995b}.

\section{Summability problem: the bivariate case}\label{SECT:criteria}
In this section, we will view rational functions in~$\bF(x, y)$
as univariate rational functions in~$y$
over the field~$\overline{\bF(x)}$.
To this end, we need to extend the summability
in~$\bF(x, y)$ to its algebraic closure $\overline{\bF(x, y)}$.
Let~$\phi, \varphi$ be two automorphisms of~$\bF(x, y)$.
Abusing notation, we still let~$\phi, \varphi$ denote the
arbitrary extensions of~$\phi, \varphi$ to~$\overline{\bF(x,y)}$.
An algebraic function~$f \in \overline{\bF(x, y)}$, is said
to be \emph{$(\phi, \varphi)$-summable in~$\overline{\bF(x, y)}$} if there exist~$g, h\in \overline{\bF(x, y)}$
such that~$f = \phi(g)-g + \varphi(h)-h$.
For a rational function~$f\in \bF(x, y)$, we will show that the summability in~$\bF(x, y)$ and
that in~$\overline{\bF(x, y)}$ are equivalent. To this end, we will need the following lemma.
\begin{lemma}\label{lem:sum} Let~$k$ be a field of characteristic zero and
let~$\bark$ be its algebraic closure and let~$\theta :\bark\rightarrow \bark$
be an automorphism such that~$\theta(k) = k$.  Let~$\alpha \in \bark$ and
let~$K, k\subset K \subset \bark$, be a finite normal extension of~$k$
containing~$\alpha$ and~$\theta(\alpha)$. If~${\text Tr}_{K/k}$ denotes
the trace, then for~$\alpha \in K$
\[{\text Tr}_{K/k}(\theta(\alpha)) = \theta({\text Tr}_{K/k}(\alpha)).\]
\end{lemma}
\begin{proof}
Let $P(z) = z^m + p_1z^{m-1} + \ldots + p_m \in k[z]$
be the minimum polynomial of~$\alpha$ over $k$.
Note that $\text{Tr}_{K/k}(\alpha) = -n p_1$,
where $n = [K:k(\alpha)]$. Furthermore note that the minimum polynomial
of $\theta(\alpha)$ is~$P^\theta(z) = z^m + \theta(p_1)z^{m-1} + \ldots + \theta(p_m)$.
Therefore ${\text Tr}_{K/k}(\theta(\alpha)) = -n \theta(p_1) = \theta(-n p_1) = \theta({\text Tr}_{K/k}(\alpha))$.
\QED
\end{proof}

\begin{theorem}\label{THM:sum2}
Let~$f\in \bF(x, y)$ and assume that~$\phi, \varphi$ be two $\bF$-automorphisms of~$\bF(x, y)$.
Then $f$ is $(\phi, \varphi)$-summable in~$\overline{\bF(x,y)}$ if and only if
$f$ is $(\phi, \varphi)$-summable in~$\bF(x, y)$.
\end{theorem}
\begin{proof}
The sufficiency is obvious. Conversely, assume that~$f$ is $(\phi, \varphi)$-summable in~$\overline{\bF(x,y)}$,
that is there exist~$g, h\in \overline{\bF(x, y)}$ such that
\[f = \phi(g) - g + \varphi(h) - h.\]
Let~$k=\bF(x, y)$ and~$\bar k = \overline{\bF(x, y)}$.
Let $K, k \subset K \subset \bark$ be a finite normal extension of~$k$
containing~$g, h, \phi(g), \varphi(h)$. Applying Lemma~\ref{lem:sum} to $ \theta = \phi, \alpha = g$
yields
\[{\text Tr}_{K/k}(\phi(g)) = \phi({\text Tr}_{K/k}(g)).\]
Similarly we deduce that ${\text Tr}_{K/k}(\varphi(h)) = \varphi({\text Tr}_{K/k}(h))$. Therefore,
for the integer~$N = [K:k]$, we have
\begin{eqnarray*} N f & = & {\text Tr}_{K/k}(\phi(g) - g + \varphi(h) - h)\\
& = & {\text Tr}_{K/k}(\phi(g)) - {\text Tr}_{K/k}(g) + {\text Tr}_{K/k}(\varphi(h)) - {\text Tr}_{K/k}(g)\\
& = &\phi({\text Tr}_{K/k}(g)) -  {\text Tr}_{K/k}(g) +\varphi( {\text Tr}_{K/k}(h))-{\text Tr}_{K/k}(g).
\end{eqnarray*}
Since ${\text Tr}_{K/k}(g)$ and~${\text Tr}_{K/k}(h)$ are in~$k$, we have shown that~$f$ is $(\phi, \varphi)$-summable in $k$.
\QED
\end{proof}
The following fact, together with Fact~\ref{FACT:1}, will be used to simplify the summability problem.
\begin{fact}\label{FACT:2}
Let~$\phi$ be an automorphism of~$\overline{\bF(x)}(y)$ such that~$\phi(y)=y$ and let~$\alpha, \beta \in \overline{\bF(x)}$.
Then for all~$m, n\in \bN$ we have
\begin{equation*}
\frac{\phi^n(\alpha)}{(y- \phi^n(\beta))^m} = \phi(g)-g + \frac{\alpha}{(y-\beta)^m}, \quad
\text{where~$g=\sum_{j=0}^{n-1} \frac{\phi^{j}(\alpha)}{(y-\phi^j(\beta))^m}$}.
\end{equation*}
\end{fact}

\subsection{The shift case}
In this case, we consider the problem of deciding whether a given rational function~$f\in \bF(x, y)$
is equal to~$\si_x(g)-g + \si_y(h) - h$ for some~$g, h\in \bF(x, y)$. By Theorem~\ref{THM:sum2}, this problem is equivalent to that of
deciding whether~$f$ is~$(\si_x, \si_y)$-summable in~$\overline{\bF(x)}(y)$.

\begin{lemma}\label{LM:dsf}
Let~$f$ be a rational function in~$\bF(x, y)$. Then~$f$ can be decomposed into~$f = \si_x(g) - g + \si_y(h) - h + r$,
where~$g, h \in \overline{\bF(x)}(y)$ and~$r$ is of the form
\begin{equation} \label{EQ:r}
r = \sum_{i=1}^m \sum_{j=1}^{n_i} \frac{\alpha_{i, j}}{(y-\beta_i)^j}
\end{equation}
with~$\alpha_{i, j}, \beta_i\in \overline{\bF(x)}$, $\alpha_{i, j}\neq 0$, and for all~$i, i^\prime$ with~$1\leq i < i^\prime \leq m$
\begin{equation}\label{EQ:cond}
\beta_i-\si_x^n(\beta_{i'})\notin \bZ \quad \text{ for any~$n\in \bZ$.}
\end{equation}
Moreover, the rational function~$f$ is $(\si_x, \si_y)$-summable in~$\overline{\bF(x)}(y)$ if and only if
the function~$r\in \overline{\bF(x)}(y)$ is~$(\si_x, \si_y)$-summable in~$\overline{\bF(x)}(y)$.
\end{lemma}
\begin{proof}
Let~$\bE = \overline{\bF(x)}$. According to the discussion in~Section~\ref{SUBSECT:ushift}, there exist~$\bar g, \bar{r}\in \bE(y)$
such that~$f = \si_y(\bar{g}) - \bar{g} + \bar{r}$, where~$\bar{r}$ is of the form
\[\bar{r}=\sum_{i=1}^m\sum_{j=1}^{n_i} \frac{\bar{\alpha}_{ij}}{(y-\beta_i)^j}
\]
with~$\bar{\alpha}_{ij}, \beta_i\in \overline{\bE}$
and~$\beta_i$'s being in distinct $\bZ$-orbits. Assume that for some~$i, i^\prime$ with~$1\leq i < i^\prime \leq m$, we
have~$\beta_i-\si_x^t(\beta_{i'}) = s \in \bZ$. Since~$\beta_i$ and~$\beta_{i^\prime}$ are in distinct $\bZ$-orbits, we have~$t\neq 0$.
By Facts~\ref{FACT:1} and~\ref{FACT:2}, there exist~$g_{i, j}, h_{i, j}\in \bE(y)$ such that
\[\frac{\bar \alpha_{i, j}}{(y-\beta_i)^j} - \frac{\bar \alpha_{i', j}}{(y-\beta_{i'})^j} = \si_x(g_{i, j})-g_{i, j} + \si_y(h_{i, j}) -h_{i, j} +
\frac{\si_x^{-t}(\bar \alpha_{i, j})-\bar \alpha_{i', j}}{(y-\beta_{i'})^j}.\]
This allows us to eliminate a term and we can repeat this process until the~$\beta_i$'s satisfy the condition~\eqref{EQ:cond}.
The remaining equivalence is obvious.
\QED
\end{proof}

\begin{lemma}\label{LM:simplefrac}
Let~$\alpha, \beta\in \overline{\bF(x)}$. If~$\beta=\frac{s}{t} x + c$ with~$s\in \bZ$, $t\in \bN\setminus \{0\}$ and~$c\in {\bF}$
and~$\alpha = \si_x^t(\gamma)-\gamma$ for some~$\gamma\in  \overline{\bF(x)}$, then the fraction~$\frac{\alpha}{(y-\beta)^j}$ is
$(\si_x, \si_y)$-summable in~$\overline{\bF(x)}(y)$.
\end{lemma}
\begin{proof}
Let
\[g = \sum_{\ell=0}^{t-1} \frac{\si_x^\ell(\gamma)}{(y-\si_x^\ell(\beta))^j}.\]
Then
\begin{align*}
r = \frac{\alpha}{(y-\beta)^j}-(\si_x(g)-g) & = \frac{\alpha}{(y-\beta)^j} - \left(\frac{\si_x^{t}(\gamma)}{(y-\si_x^{t}(\beta))^j} - \frac{\gamma}{(y-\beta)^j}\right)
\\
   & =\frac{\alpha + \gamma}{(y-\beta_j)^j} - \frac{\si_x^{t}(\gamma)}{(y-\si_x^{t}(\beta))^j}.
\end{align*}
Note that~$\si_x^{t}(\beta) - \beta= s\in \bZ$. Since~$\alpha = \si_x^{t}(\gamma) -\gamma$, we have the discrete
residue of~$r$ at~$\beta$ of multiplicity~$j$ is zero. Then Proposition~\ref{PROP:ratsum} implies that
there exists~$h\in \overline{\bF(x)}(y)$ such that
\[\frac{\alpha}{(y-\beta)^j} = \si_x(g)-g + \si_y(h) - h,\]
which completes the proof. \QED
\end{proof}

We recall a lemma from~\cite{ChenSinger2012}.
\begin{lemma}[c.f. Lemma 3.7 in~\cite{ChenSinger2012}]\label{LM:intlin}
Let~$\alpha(x)$ be an element in the algebraic closure of~$\bF(x)$.
If there exists a nonzero~$n\in \bZ$ such that~$\si_x^n(\alpha)-\alpha = m$ for some~$m\in \bZ$,
then~$\alpha(x)=\frac{m}{n} x + c$ for some~$c\in \bF$.
\end{lemma}

Note that the shift operators~$\si_x$ and~$\si_y$ preserve the multiplicities of irreducible factors in the denominators of
rational functions.
Therefore the rational function~$r$ in~\eqref{EQ:r} is $(\si_x, \si_y)$-summable in~$\overline{\bF(x)}(y)$ if and only if for each~$j$,
the rational function
\begin{equation}\label{EQ:rj}
r_j = \sum_{i=1}^m \frac{\alpha_{ij}}{(y-\beta_i)^j}
\end{equation}
is $(\si_x, \si_y)$-summable in~$\overline{\bF(x)}(y)$.

\begin{theorem}\label{THM:ratsum}
Let~$f\in \overline{\bF(x)}(y)$ be of the form~\eqref{EQ:rj} with~$\alpha_{i, j}, \beta_i$ in~$\overline{\bF(x)}$,
$\alpha_{i, j} \neq 0$, and the~$\beta_i$'s satisfying
the condition~\eqref{EQ:cond}.
Then~$f$ is $(\si_x, \si_y)$-summable in~$\overline{\bF(x)}(y)$ if and only if
for each~$i\in \{1, 2, \ldots, m\}$, we have
\[\beta_i = \frac{s_i}{t_i}x + c_i, \quad \text{where~$s_i\in \bZ$, $t_i\in\bN\setminus \{0\}$, and~$c_i\in \bF$},\]
and~$\alpha_{i, j}= \si_x^{t_i}(\gamma_i) -\gamma_i$ for some~$\gamma_i\in {\bF(x)}$.
\end{theorem}
\begin{proof} The sufficiency follows from Lemma~\ref{LM:simplefrac}.
For the necessity, we assume that~$f$ is $(\si_x, \si_y)$-summable in~$\overline{\bF(x)}(y)$, i.e., there exist~$g, h\in \overline{\bF(x)}(y)$ such that
\begin{equation}\label{EQ:thm}
f = \si_x(g) - g + \si_y(h) - h.
\end{equation}
We decompose the rational function~$g$ into the form
\begin{equation}\label{EQ:thm2}
g = \si_y(g_1)-g_1 + g_2 + \frac{\lambda_1}{(y- \mu_1)^j} + \cdots + \frac{\lambda_n}{(y-\mu_n)^j},
\end{equation}
where~$g_1, g_2\in \bF(x, y)$, $g_2$ is a rational function
having no terms of the form $1/(y-\nu)^j$ in its partial fraction decomposition with respect
to $y$,~$\lambda_k, \mu_k\in \overline{\bF(x)}$, and the~$\mu_k$'s are in distinct~$\bZ$-orbits.

\medskip
\noindent {\bf Claim 1.} For each~$i\in \{1, 2, \ldots, m\}$, at least one element of the set
\[\Lambda := \{\mu_1, \ldots, \mu_n, \si_x(\mu_1), \ldots, \si_x(\mu_n)\}\]
is in the same $\bZ$-orbit as~$\beta_i$. For each~$\eta \in \Lambda $, there is  one
element of  $ \Lambda\backslash \{\eta\} \cup \{ \beta_1, \ldots , \beta_m\}$ that is $\bZ$-equivalent to~$\eta$.

\medskip
\noindent {\bf Proof of Claim 1.}
Suppose no element of~$\Lambda$ is in the same $\bZ$-orbit
as~$\beta_i$.  Taking the discrete residues on both sides of~\eqref{EQ:thm}, we get~$\dres_y(f, \beta_i, j) = \alpha_{i, j}\neq 0$
and~$\dres_y(\si_x(g)-g + \si_y(h)-h, \beta_i, j) =0$, which is a contradiction. The second assertion follows from
the same argument.

\medskip
Claim 1 implies that either $\beta_i \sim_{\bZ} \mu'_1$ or $\beta_i \sim_{\bZ} \si_x(\mu'_1)$ for some $\mu'_1 \in \{\mu_1, \ldots, \mu_n\}$. We shall deal with each case separately.

\medskip
{\noindent {\bf Claim 2.} Assume $\beta_i \sim_{\bZ} \mu'_1$.\\[0.05in]
 (a)  Fix  $\beta_i$ and $j \in \bN, j \geq 2$ and assume that $\si_x^{\ell} \beta_i \not\sim_\bZ \beta_i$ for $1 \leq \ell \leq j-1$. Then there exist $\mu'_1, \ldots \mu'_j \in \{\mu_1, \ldots, \mu_n\}$ such that
\begin{equation}\label{EQ:claim2a}
\si_x^{j-1}(\beta_i)\sim_{\bZ}\mu'_j, \text{ and }
\end{equation}
\begin{equation}\label{EQ:claim2b}
\si_x(\mu'_{1})\sim_{\bZ} \mu'_{2}, \, \, \si_x(\mu'_{2})\sim_{\bZ} \mu'_{3},  \, \, \ldots,  \, \, \si_x(\mu'_{j-1}) \sim_{\bZ} \mu'_{{j}}.
\end{equation}
(b) There exists~$t_i\in \bN, t_i \neq 0$ such that~$t_i\leq n$ and~$\si_x^{t_i}(\beta_i)-\beta_i\in \bZ$. For the smallest such $t_i$, there exist $\mu'_1, \ldots \mu'_{t_i-1} \in \{\mu_1, \ldots, \mu_n\}$ such that
\begin{equation}\label{EQ:claim2c}
\si_x(\mu'_{1})\sim_{\bZ} \mu'_{2}, \, \, \si_x(\mu'_{2})\sim_{\bZ} \mu'_{3},  \, \, \ldots,  \, \, \si_x(\mu'_{t_i-1}) \sim_{\bZ} \mu'_{{t_i}}, \si_x(\mu'_{t_i}) \sim_{\bZ} \beta_i.
\end{equation}

\medskip
\noindent {\bf Proof of Claim 2.}
(a)  Let us first assume that $j = 2$. From the second part of Claim 1, we have  that $\si_x(\mu'_1)$ is $\bZ$-equivalent to an element of   $ \Lambda\backslash \{\si_x(\mu'_1)\} \cup \{ \beta_1, \ldots , \beta_m\}$. If $\si_x(\mu'_1)$ is $\bZ$-equivalent to some $\beta_\ell$ for $\ell \neq i$, then $\si_x(\beta_\ell) \sim_\bZ \beta_i$, contradicting \eqref{EQ:cond}. If $\si_x(\mu'_1)$ is $\bZ$-equivalent to  $\beta_i$,
then we would have $\si_x(\beta_i) \sim_\bZ  \si_x(\mu'_i) \sim_\bZ \beta_i$, contradicting} the assumption of Claim 2(a).
{If $\si_x(\mu'_1)\sim_\bZ \si_x(\mu_\ell)$ for some $\mu_\ell \neq \mu'_1$,
then $\mu_\ell \sim_\bZ \mu'_1$, contradiction our assumption that the $\mu_i$ are in distinct $\bZ$-orbits.
Therefore we are left with only one possibility - that $\si_x(\mu'_1) \sim_\bZ \mu'_2$ for some $\mu'_2$
and \eqref{EQ:claim2a} and \eqref{EQ:claim2b} hold for this choice.   Now assume that \eqref{EQ:claim2a}
and \eqref{EQ:claim2b} hold for $j > 2$. Arguing as in the case when $j = 2$, we can verify that there
exists a $\mu'_{j+1}$ such that  \eqref{EQ:claim2a} and \eqref{EQ:claim2b} hold in this case as well.\\[0.1in]
 (b) If such a $t_i$ does not exist, then one could find $\{\mu'_1, \ldots , \mu'_{n+1}\}$ satisfying  \eqref{EQ:claim2a} and \eqref{EQ:claim2b}. In this case, we must have $\mu'_r = \mu'_s$ for some $r>s$ implying that $\si_x^r \beta_i \sim_\bZ \si_x^s \beta_i$. This implies $\si_x^{r-s} \beta_i \sim_\bZ \beta_i$ a contradiction. Therefore the first part of (b) is verified. To verify the second part, apply part (a) to $j = t_i$.

\medskip
\noindent {\bf Claim 3.} Assume $\beta_i \sim_\bZ \si_x(\mu'_1)$.\\[0.05in]
(a)  Fix  $\beta_i$ and $j \in \bN, j \geq 2$ and assume that $\si_x^{\ell} \beta_i \not\sim_\bZ \beta_i$ for $1 \leq \ell \leq j-1$. Then there exist $\mu'_1, \ldots \mu'_j \in \{\mu_1, \ldots, \mu_n\}$ such that
\begin{equation}\label{EQ:claim3a}
\beta_i\sim_{\bZ}\si_x^{j}(\mu'_j), \text{ and }
\end{equation}
\begin{equation}\label{EQ:claim3b}
\mu'_{1}\sim_{\bZ} \si_x(\mu'_{2}), \, \, \mu'_{2}\sim_{\bZ}\si_x( \mu'_{3}),  \, \, \ldots,  \, \, \mu'_{j-1}\sim_{\bZ} \si_x(\mu'_{{j}}).
\end{equation}
(b) There exists~$t_i\in \bN$ such that~$t_i\leq n$ and~$\si_x^{t_i}(\beta_i)-\beta_i\in \bZ$. For the smallest such $t_i$, there exist $\mu'_1, \ldots \mu'_{t_i} \in \{\mu_1, \ldots, \mu_n\}$ such that
\begin{equation}\label{EQ:claim3c}
\mu'_{1}\sim_{\bZ} \si_x(\mu'_{2}), \, \, \mu'_{2}\sim_{\bZ} \si_x(\mu'_{3}),  \, \, \ldots,  \, \, \mu'_{t_i-1} \sim_{\bZ} \si_x(\mu'_{t_i}), \mu'_{t_i} \sim_\bZ \beta_i.
\end{equation}

\medskip
\noindent {\bf Proof of Claim 3.} (a) Let us first assume that $j = 2$.    Again from the second part of Claim 1, we have that $\mu'_1$ is $\bZ$-equivalent to an element of   $ \Lambda\backslash \{\mu'_1\} \cup \{ \beta_1, \ldots , \beta_m\}$. If $\mu'_1$ is $\bZ$-equivalent to some $\beta_\ell$ for $\ell \neq i$, then $\si_x(\beta_\ell) \sim_\bZ \beta_i$, contradicting \eqref{EQ:cond}. If $\mu'_1$ is $\bZ$-equivalent to  $\beta_i$, we contradict the assumption of Claim 2.  If $\mu'_1\sim_\bZ \mu_\ell$ for some $\mu_\ell \neq \mu'_1$, we contradict our assumption that the $\mu_i$ are in distinct $\bZ$-orbits. Therefore we are left with only one possibility - that $\mu'_1 \sim_\bZ \si_x(\mu'_2)$ for some $\mu'_2$.  Therefore \eqref{EQ:claim3a} and \eqref{EQ:claim3b} hold in this case as well. Now assume that \eqref{EQ:claim3a} and \eqref{EQ:claim3b} hold for $j > 2$. Arguing as in the case when $j = 2$, we can verify that there exists a $\mu'_{j+1}$ such that  \eqref{EQ:claim3a} and \eqref{EQ:claim3b} hold in this case as well.\\[0.1in]
(b) If such a $t_i$ does not exist, then one could find $\{\mu'_1, \ldots , \mu'_{n+1}\}$ satisfying  \eqref{EQ:claim3a} and \eqref{EQ:claim3b}. In this case, we must have $\mu'_r = \mu'_s$ for some $r>s$ implying that $\si_x^{-r} \beta_i \sim_\bZ \si_x^{-s} \beta_i$. This implies $\si_x^{r-s} \beta_i \sim_\bZ \beta_i$ a contradiction. Therefore the first part of (b) is verified. To verify the second part, apply part (a) to $j = t_i$. \\[0.2in]
 Using these claims, we now complete the proof. From Claims 2(b) and 3(b), we have that for each $i$ there exists a positive integer $t_i$ such that $\si_x^{t_i}(\beta_i)-\beta_i\in \bZ$. This implies that~$\beta_i = \frac{s_i}{t_i} x + c_i$ for some~$s_i\in \bZ$ and~$c_i\in \bF$
by Lemma~\ref{LM:intlin}. We now turn to verifying the claim of the Theorem concerning the $\alpha_{i,j}$.\\[0.05in]
Fix some $\beta_i$ and assume, as in Claim 2,  that~$\beta_i - \mu'_1 \in \bZ$. We wish to
compare the discrete residues at  $\beta_i$ on the left side of~\eqref{EQ:thm2} with the
discrete residues at the elements of~$\Lambda$ on the right side of~\eqref{EQ:thm2}.
The equivalences of \eqref{EQ:claim2c} give the $\bZ$-orbits in $\Lambda$.
In the following table, the first column lists the $\bZ$-orbits of elements in~$\Lambda$.
The second column equates the discrete residue of this orbit on the left of~\eqref{EQ:thm2}
with the discrete residue of the same orbit on the right of~\eqref{EQ:thm2}.
Note that the orbit listed on the first line corresponds to $\beta_i$ and that the 
other orbits have zero residue on the left of~\eqref{EQ:thm2}.
\begin{center}
\begin{tabular}{|c|c|}  \hline
$\bZ$-orbit &  Comparison of two sides of \eqref{EQ:thm2}\\ \hline
$\mu'_1, \si_x(\mu'_{t_i})$ & $\alpha_{i,j} \ = \ \si_x(\lambda_{t_i}) - \lambda_1$ \\ \hline
$\mu'_{t_i}, \si_x(\mu'_{t_i-1}) $ & $0  \ = \  \si_x(\lambda_{t_i-1}) - \lambda_{t_i}$. \\ \hline
$\mu'_{t_i-1}, \si_x(\mu'_{t_i-2}) $ & $0  \ = \  \si_x(\lambda_{t_i-2}) - \lambda_{t_i-1}$. \\ \hline
  \vdots      &  \vdots     \\ \hline
    $\mu'_3, \si_x(\mu'_{2})$ & $0        \ = \ \si_x(\lambda_2) - \lambda_3$ \\\hline
  $\mu'_2, \si_x(\mu'_{1})$ & $0        \ = \ \si_x(\lambda_1) - \lambda_2$ \\\hline

\end{tabular}
\end{center}
Using the equations in the last column, to eliminate all intermediate terms one can show that $\alpha_{i,j}  =  \si_x^{t_i}(\lambda_1) - \lambda_1$. Since~$\beta_i\in \bF(x)$, ~$\mu_1-\beta_i\in \bZ$, and~$\lambda_1\in \bF(\mu_1)$,
the element~$\lambda_1$ is actually in~$\bF(x)$.\\[0.1in]
We now turn to the situation of Claim 3, that is, assume that $\beta_i -\si_x(\mu'_1) \in \bZ$. As in the previous paragraph, we will compare the discrete residues at the $\beta_i$ on the left side of~\eqref{EQ:thm} with the discrete residues at the elements of~$\Lambda$ on the right side of~\eqref{EQ:thm}. The equivalences of \eqref{EQ:claim3c} give the $\bZ$-orbits.  The following  table summarizes the comparison.

\medskip
\begin{center}
\begin{tabular}{|c|c|}  \hline
$\bZ$-orbit &  Comparison of two sides of \eqref{EQ:thm2}\\ \hline
$\si_x(\mu'_1), \mu'_{t_i}$ & $\alpha_{i,j} \ = \ \si_x(\lambda_{1}) - \lambda_{t_i}$ \\ \hline
$ \si_x(\mu'_{2}), \mu'_{1}$ & $0  \ = \  \si_x(\lambda_{2}) - \lambda_{1}$. \\ \hline
$ \si_x(\mu'_{3}), \mu'_{2}$  & $0  \ = \  \si_x(\lambda_{3}) - \lambda_{2}$. \\ \hline
\vdots      &  \vdots     \\ \hline
$ \si_x(\mu'_{t_i-1}), \mu'_{t_i-2}$ & $0        \ = \ \si_x(\lambda_{t_i-1}) - \lambda_{t_i-2}$ \\\hline
$ \si_x(\mu'_{t_i}), \mu'_{t_i-1}$ & $0        \ = \ \si_x(\lambda_{t_i}) - \lambda_{t_i-1}$ \\\hline
\end{tabular}
\end{center}
Using the equations in the last column, to eliminate all intermediate terms one can show that $\alpha_{i,j}  =  \si_x^{t_i}(\lambda_{t_{i}}) - \lambda_{t_i}$. Since~$\beta_i\in \bF(x)$, ~$\mu_1-\beta_i\in \bZ$, and~$\lambda_1\in \bF(\mu_1)$,
the element~$\lambda_1$ is actually in~$\bF(x)$. }
\QED
\end{proof}

\begin{example}\label{EXAM:ratsum}
Let~$f = 1/(x^n+y^n)$ with~$n\in \bN\setminus \{0\}$. Over the field~$\overline{\bF(x)}$, we can decompose~$f$ into
\begin{equation}\label{EQ:exam}
  f = \sum_{i=1}^n \frac{\alpha_i}{y - \beta_i},
\end{equation}
where~$\beta_i = \omega_i x$ with~$\omega_i$ varying over the roots of~$z^n=-1$ and~$\alpha_i = \frac{1}{n(\omega_ix)^{n-1}}$.
By Theorem~\ref{THM:ratsum}, $f$ is $(\si_x, \si_y)$-summable in~$\bF(x, y)$ if and
only if for all~$i\in \{1, \ldots, n\}$, ~$\omega_i =s_i/t_i$ for some~$s_i\in \bZ$, $t_i\in\bN\setminus \{0\}$,
and~$\alpha_i= \si_x^{t_i}(\gamma_i) -\gamma_i$ for some~$\gamma_i\in {\bF(x)}$.
When~$n>1$, at least one~$\omega_i$ is not a rational number, which implies that~$f$ is not $(\si_x, \si_y)$-summable
in~$\bF(x, y)$.  When~$n=1$,  the discrete residue of~$f$ at~$-x$ is $1$ and~$1=\si_x(x)-x$. Therefore, $f=1/(x+y)$ is
$(\si_x, \si_y)$-summable in~$\bF(x, y)$. In fact, we have
\[\frac{1}{x+y} = \si_x\left( \frac{x}{x+y}\right)  - \frac{x}{x+y} + \si_y\left(\frac{-x-1}{x+y}\right) - \frac{-x-1}{x+y}. \]
\end{example}

\begin{example} Let~$f = 1/xy$. The discrete residue of~$f$ at~$0$ of multiplicity one is~$1/x$. Since~$1/x \neq \si_x(\lambda) -\lambda$ for
any~$\lambda \in \bF(x)$,  $f$ is not $(\si_x, \si_y)$-summable in~$\bF(x, y)$.
\end{example}


\begin{example}\label{EXAM:tornheim}
The harmonic double sums
\[T(r, s, t) = \sum_{n=1}^{\infty}\sum_{m=1}^{\infty} \frac{1}{n^rm^s(n+m)^t}\]
were studied by Tornheim~\cite{Tornheim1950, Tornheim1950b} and Mordell~\cite{Mordell1960}
and many elegant identities have been established between them~\cite{Huard1996, Moll2006, Moll2010, Basu2011}.
Tornheim~\cite[Thm.\ 5]{Tornheim1950} proved that
\begin{equation}\label{EQ:identity}
T(0, 0, t) = \zeta(t-1) - \zeta(t), \quad \text{where~$t>2$ and $\zeta(s) = \sum_{n=1}^{\infty}\frac{1}{n^s}$.}
\end{equation}
We give another proof as follows. Let~$f = 1/(n+m)^t$.
Let~$\si_n, \si_m$ be the shift operators with respect to~$n$ and~$m$, respectively.
Set~$\Delta_n = \si_n-1$ and~$\Delta_m = \si_m - 1$.
By Lemma~\ref{LM:simplefrac}, $f$ is $(\si_n, \si_m)$-summable in~$\bQ(n, m)$.
In fact, we have
\[\frac{1}{(n+m)^t} = \Delta_n\left(\frac{n}{(n+m)^t}\right) + \Delta_m\left(\frac{-n-1}{(n+m)^t}\right).\]
Since
\begin{align*}
  \sum_{n=1}^{\infty}\sum_{m=1}^{\infty} \Delta_n\left(\frac{n}{(n+m)^t}\right)
 &= \sum_{m=1}^{\infty} \left(\sum_{n=1}^{\infty}\Delta_n\left(\frac{n}{(n+m)^t}\right)\right)  \\
   & = \sum_{m=1}^{\infty}\frac{-1}{(1+m)^t} = 1-\zeta(t),
\end{align*}
and
\begin{align*}
\sum_{n=1}^{\infty}\sum_{m=1}^{\infty} \Delta_m\left(\frac{-n-1}{(n+m)^t}\right)
 &= \sum_{n=1}^{\infty} \left(\sum_{m=1}^{\infty}\Delta_m\left(\frac{-n-1}{(n+m)^t}\right)\right) \\
  & = \sum_{n=1}^{\infty}\frac{1}{(1+n)^{t-1}} = \zeta(t-1)-1.
\end{align*}
This completes the proof of the identity~\eqref{EQ:identity}.
\end{example}
\begin{example}
We show the identity
\[\sum_{n=1}^{\infty}\sum_{m=1}^{\infty} \frac{1}{(m+n/2)^3} = 4\zeta(2)-\frac{9}{2}\zeta(3).\]
By Lemma~\ref{LM:simplefrac}, we have
\begin{equation}\label{EQ:identity2}
\frac{1}{(m + n/2)^3} = \Delta_n\left(\frac{n/2}{(m+n/2)^3} + \frac{(n+1)/2}{(m+(n+1)/2)^3}\right)+ \Delta_m\left(\frac{-1-n/2}{(m+n/2)^3}\right).
\end{equation}
Summing both sides of~\eqref{EQ:identity2} with respect to~$n$ and~$m$ yields
\[\sum_{n=1}^{\infty}\sum_{m=1}^{\infty} \frac{1}{(m+n/2)^3} = \sum_{m=1}^{\infty} \left(\frac{-4}{(2m+1)^3} + \frac{-1}{(m+1)^3}\right)
+\sum_{n=1}^{\infty}\frac{4}{(n+2)^2}. \]
Note that
\[ \sum_{m=1}^{\infty} \frac{1}{(2m+1)^3}= \sum_{m=1}^{\infty} \frac{1}{m^3} -1 - \sum_{m=1}^{\infty} \frac{1}{(2m)^3} = \frac{7}{8}\zeta(3)-1,\]
which implies
\[\sum_{n=1}^{\infty}\sum_{m=1}^{\infty} \frac{1}{(m+n/2)^3} = -\frac{7}{2}\zeta(3)+4 - \zeta(3) +1 + 4\zeta(2) - 4- 1=  4\zeta(2)-\frac{9}{2}\zeta(3).\]
\end{example}

\subsection{The $q$-shift case}
In this case, we consider the problem of deciding whether a given rational function~$f\in \bF(x, y)$
is equal to~$\tau_x(g)-g + \tau_y(h) - h$ for some~$g, h\in \bF(x, y)$. By Theorem~\ref{THM:sum2}, this problem is equivalent to that of
deciding whether~$f$ is~$(\tau_x, \tau_y)$-summable in~$\overline{\bF(x)}(y)$.

{We first consider the case in which~$q$ is a root of unity. Let~$m$ be the minimal positive integer such
that~$q^m=1$.  One can show that   $\tau_x(f) = \tau_y(f) = f$ if and only if $f \in \bF(x^m, y^m)$. Furthermore, $[\bF(x, y):\bF(x,y^m)]= m$ and
$[\bF(x, y^m):\bF(x^m,y^m)]= m$ and therefore  $[\bF(x, y):\bF(x^m,y^m)]= m^2$. The field $\bF(x, y)$ is a Galois extension of $\bF(x^m, y^m)$ whose Galois group is the product of the cyclic group generated by $\tau_x$ and the cyclic group generated by $\tau_y$. We now derive a normal form for functions in $\bF(x,y)$ with respect to $\tau_y$ and $\tau_x$.
\begin{lemma}\label{LM:qqru}
Let $q \in \bF$ be such that  $q^m = 1$ with $m$ minimal.  Let $f \in \bF(x,y)$.
\begin{enumerate}
\item[(a)] $f = \tau_x(g) - g + \tau_y(h) - h$ for some $g,h \in \bF(x,y)$ if and only if \linebreak $\Tr_{\bF(x,y)/\bF(x^m,y^m)}(f) = 0$.
\item[(b)] Any rational function~$f\in \bF(x,y)$ can be decomposed into
\begin{equation}\label{EQ:fnq2}
f = \tau_x(g) - g  + \tau_y(h)-h+ c, \quad \text{where~$g,h\in \bF(x,y)$ and~$c\in \bF(x^m,y^m)$}.
\end{equation}
Moreover, $f$ is~$(\tau_x,\tau_y)$-summable if and only if~$c=0$.
\end{enumerate}
\end{lemma}

\begin{proof} (a) Assume $f = \tau_x(g) - g + \tau_y(h) - h$. Applying Lemma~\ref{LM:qru}(a) to $f-\tau_x(g) - g$ we have
\begin{align*}
  0 & = \Tr_{\bF(x,y)/\bF(x,y^m)} (f-(\tau_x(g) - g))\\
   & = \Tr_{\bF(x,y)/\bF(x,y^m)}(f) - \Tr_{\bF(x,y)/\bF(x,y^m)}(\tau_x(g) - g).
\end{align*}
Lemma~\ref{lem:sum} implies \[\Tr_{\bF(x,y)/\bF(x,y^m)}(\tau_x(g) - g) = \tau_x(\Tr_{\bF(x,y)/\bF(x,y^m)}(g)) - \Tr_{\bF(x,y)/\bF(x,y^m)}(g).\]  Therefore \[\Tr_{\bF(x,y)/\bF(x,y^m)}(f) =  \tau_x(\Tr_{\bF(x,y)/\bF(x,y^m)}(g)) - \Tr_{\bF(x,y)/\bF(x,y^m)}(g).\]  Applying  Lemma~\ref{LM:qru}(a), we have $\Tr_{\bF(x,y^m)/\bF(x^m,y^m)}(\Tr_{\bF(x,y)/\bF(x,y^m)}(f)) = 0$. Since $\Tr_{\bF(x,y)/\bF(x,y)} = \Tr_{\bF(x,y^m)/\bF(x^m,y^m)}\circ \Tr_{\bF(x,y)/\bF(x,y^m)}$ (see~\cite[Thm. 5.1, p.~285]{LANG}), we have $\Tr_{\bF(x,y)/\bF(x^m,y^m)}(f) = 0$.\\
Now assume that
\[0=\Tr_{\bF(x,y)/\bF(x^m,y^m)}(f) = \Tr_{\bF(x,y^m)/\bF(x^m,y^m)}(\Tr_{\bF(x,y)/\bF(x,y^m)}(f)).\]
Lemma~\ref{LM:qru}(a) implies that $\Tr_{\bF(x,y)/\bF(x,y^m)}(f) = \tau_x(g) - g$ for
some $g \in \bF(x,y^m)$.  Note that
\begin{eqnarray*}
 \Tr_{\bF(x,y)/\bF(x,y^m)}(mf -\Tr_{\bF(x,y)/\bF(x,y^m)}(f))& =&\\ m\Tr_{\bF(x,y)/\bF(x,y^m)}(f) - m \Tr_{\bF(x,y)/\bF(x,y^m)}(f) &=& 0.
\end{eqnarray*}
Therefore, Lemma~\ref{LM:qru}(a) implies $mf -\Tr_{\bF(x,y)/\bF(x,y^m)}(f)) = \tau_y(h) - h$  for some $h \in \bF(x,y)$.  Therefore,
\[f = \frac{1}{m}(\tau_x(g) - g + \tau_y(h) - h).\]
(b) We can write
\[ f = \left(\sum_{0\leq i,j \leq m-1, (i,j) \neq (0,0)}a_{i,j}x^iy^j\right) + a_{0,0},\]
where all $a_{i,j} \in \bF(x^m,y^m)$.  A calculation shows that $\Tr_{\bF(x,y)/\bF(x^m,y^m)}(x^iy^j)= 0$ when $0\leq i,j \leq m-1, (i,j) \neq (0,0)$.
Therefore part (a) implies  that $f = \tau_x(g) - g  + \tau_y(h)-h+ a_{0,0}$. Therefore $f$ is~$(\tau_x,\tau_y)$-summable if and
only if the trace $\Tr_{\bF(x,y)/\bF(x^m,y^m)}(a_{0,0}) = 0$ but this is true if and only if $a_{0,0} = 0$.
\QED
\end{proof}
Similar to the comment following Lemma~\ref{LM:qru}(a), one sees that it is easy to verify
the~$(\tau_x,\tau_y)$-summability of any $f \in \bF(x,y)$.

\begin{example}\label{EXAM:qratsumru}
Let~$f = 1/(x^n+y^n)$ with~$n\in \bN\setminus \{0\}$.
Recall that~$m$ is the minimal positive integer such that~$q^m=1$.
Write~$n = qm + r$ with~$0\leq r< m$. Then~$x^n + y^n = x^r (x^m)^q + y^r (y^m)^q$.
This implies that~$f\in \bF(x^m, y^m)$ if and only if~$r=0$.
Since~$f$ is nonzero, we have~$f$ is not~$(\tau_x, \tau_y)$-summable
in~$\bF(x, y)$ when~$r=0$ by Lemma~\ref{LM:qqru}. In the case when~$r\neq 0$, we have
\[
\frac{1}{x^n+y^n} = \tau_x\left(\frac{c_n}{x^n + y^n}\right)- \frac{c_n}{x^n + y^n} +
\tau_y\left(\frac{c_n}{q^nx^n + y^n}\right)- \frac{c_n}{q^nx^n + y^n},
\]
where~$c_n = q^n/(1-q^n)=q^r/(1-q^r)$, which means that~$f$ is~$(\tau_x, \tau_y)$-summable
in~$\bF(x, y)$ in this case.
\end{example}


}

From now on, we assume that~$q$ is not a root of unity.
\begin{lemma}\label{LM:qsf}
Let~$f$ be a rational function in~$\bF(x, y)$.
Then~$f$ can be decomposed into~$f = \tau_x(g) - g + \tau_y(h) - h + r$,
where~$g, h \in \overline{\bF(x)}(y)$ and~$r$ is of the form
\begin{equation} \label{EQ:qr}
r = c + \sum_{i=1}^m \sum_{j=1}^{n_i} \frac{\alpha_{i, j}}{(y-\beta_i)^j},
\end{equation}
with~$c\in \bF(x)$, $\alpha_{i, j}, \beta_i\in \overline{\bF(x)}$, $\alpha_{i, j}\neq 0$, $\beta_i \neq 0$ and for all~$i, i^\prime$ with~$1\leq i < i^\prime \leq m$
\begin{equation}\label{EQ:qcond}
\frac{\beta_i}{\tau_x^n(\beta_{i'})}\notin q^\bZ \quad \text{ for any~$n\in \bZ$.}
\end{equation}
Moreover, the rational function~$f$ is $(\tau_x, \tau_y)$-summable in~$\overline{\bF(x)}(y)$ if and only if
the function~$r\in \overline{\bF(x)}(y)$ is~$(\tau_x, \tau_y)$-summable in~$\overline{\bF(x)}(y)$.
\end{lemma}
\begin{proof}
Let~$\bE = \overline{\bF(x)}$. According to the discussion in~Section~\ref{SUBSECT:uqshift},
there exist~$\bar g, \bar{r}\in \bE(y)$
such that~$f = \tau_y(\bar{g}) - \bar{g} + \bar{r}$, where~$\bar{r}$ is of the form
\[\bar{r}= c + \sum_{i=1}^m\sum_{j=1}^{n_i} \frac{\bar{\alpha}_{ij}}{(y-\beta_i)^j}
\]
with~$c\in \bF(x)$, $\bar{\alpha}_{ij}, \beta_i\in \overline{\bE}$
and~$\beta_i$'s being in distinct $q^\bZ$-orbits.
Assume that for some~$i, i^\prime$ with~$1\leq i < i^\prime \leq m$ and~$t\in \bZ$, we
have~$\frac{\beta_i}{\tau_x^t(\beta_{i'})} = q^s \in q^\bZ$. Since~$\beta_i$
and~$\beta_{i^\prime}$ are in distinct $q^\bZ$-orbits, we then have~$t\neq 0$.
By Facts~\ref{FACT:1} and~\ref{FACT:2}, there exist~$g_{i, j}, h_{i, j}\in \bE(y)$ such that
\[\frac{\bar \alpha_{i, j}}{(y-\beta_i)^j} - \frac{\bar \alpha_{i', j}}{(y-\beta_{i'})^j} =
\tau_x(g_{i, j})-g_{i, j} + \tau_y(h_{i, j}) -h_{i, j} +
\frac{\tau_x^{-t}(q^{-sj}\bar \alpha_{i, j})-\bar \alpha_{i', j}}{(y-\beta_{i'})^j}.\]
This allows us to eliminate a term and we can repeat this process until the~$\beta_i$'s
satisfy the condition~\eqref{EQ:qcond}. The remaining equivalence is obvious.
\QED
\end{proof}

\begin{lemma}\label{LM:qsimplefrac}
Let~$\alpha, \beta\in \overline{\bF(x)}$. If~$\beta=c x^{s/t}$ with~$s\in \bZ$, $t\in \bN\setminus \{0\}$ and~$c\in {\bF}$
and~$\alpha =q^{-sj} \tau_x^t(\gamma)-\gamma$ for some~$\gamma\in  \overline{\bF(x)}$, then the fraction~$\frac{\alpha}{(y-\beta)^j}$ is
$(\tau_x, \tau_y)$-summable in~$\overline{\bF(x)}(y)$.
\end{lemma}
\begin{proof}
Let
\[g = \sum_{\ell=0}^{t-1} \frac{\tau_x^\ell(\gamma)}{(y-\tau_x^\ell(\beta))^j}.\]
Then
\begin{align*}
r = \frac{\alpha}{(y-\beta)^j}-(\tau_x(g)-g) & =
\frac{\alpha}{(y-\beta)^j} - \left(\frac{\tau_x^{t}(\gamma)}{(y-\tau_x^{t}(\beta))^j} - \frac{\gamma}{(y-\beta)^j}\right)
\\
   & =\frac{\alpha + \gamma}{(y-\beta_j)^j} - \frac{\tau_x^{t}(\gamma)}{(y-\tau_x^{t}(\beta))^j}.
\end{align*}
Note that~$\tau_x^{t}(\beta) = q^s \beta$. Since~$\alpha = q^{-sj}\tau_x^{t}(\gamma) -\gamma$, we have the $q$-discrete
residue of~$r$ at~$\beta$ of multiplicity~$j$ is zero.
Applying Proposition~\ref{PROP:ratqsum}, we have that
there exists~$h\in \overline{\bF(x)}(y)$ such that
\[\frac{\alpha}{(y-\beta)^j} = \tau_x(g)-g + \tau_y(h) - h,\]
which completes the proof.
\QED
\end{proof}

We recall a lemma from~\cite{ChenSinger2012}, which is a $q$-analogue of Lemma~\ref{LM:intlin}.
\begin{lemma}[c.f. Lemma 3.8 in~\cite{ChenSinger2012}] \label{LM:qintlin}
Let~$\alpha(x)$ be an element in the algebraic closure of~$\bF(x)$.
If there exists a nonzero~$n\in \bZ$
such that~$\tau_x^n(\alpha)= q^m\alpha$ for some~$m\in \bZ$,
then~$\alpha(t)=c x^{\frac{m}{n}}$ for some~$c\in \bF$.
\end{lemma}

Note that the $q$-shift operators~$\tau_x$ and~$\tau_y$ preserve the polynomial part and the multiplicities of irreducible factors in the denominators of
rational functions.
Therefore, the rational function~$r$ in~\eqref{EQ:qr} is $(\tau_x, \tau_y)$-summable
in~$\overline{\bF(x)}(y)$ if and only if~$c\in \bF(x)$ is~$\tau_x$-summable in~$\bF(x)$
and for each~$j$,
the rational function
\begin{equation}\label{EQ:qrj}
r_j = \sum_{i=1}^m \frac{\alpha_{ij}}{(y-\beta_i)^j}
\end{equation}
is $(\tau_x, \tau_y)$-summable in~$\overline{\bF(x)}(y)$.

\begin{theorem}\label{THM:qratsum}
Let~$f\in \overline{\bF(x)}(y)$ be of the form~\eqref{EQ:qrj} with~$\alpha_{i, j}, \beta_i$ in~$\overline{\bF(x)}$,
$\alpha_{i, j} \neq 0$, and the~$\beta_i$'s satisfying
the condition~\eqref{EQ:qcond}.
Then~$f$ is $(\tau_x, \tau_y)$-summable in~$\overline{\bF(x)}(y)$ if and only if
for each~$i\in \{1, 2, \ldots, m\}$, we have
\[\beta_i = c x^{{s_i}/{t_i}}, \quad \text{where~$s_i\in \bZ$, $t_i\in\bN\setminus \{0\}$, and~$c_i\in \bF$},\]
and~$\alpha_{i, j}= q^{-js_i }\tau_x^{t_i}(\gamma_i) -\gamma_i$ for some~$\gamma_i\in {\bF(x^{1/t_i})}$.
\end{theorem}
\begin{proof}
The sufficiency follows from Lemma~\ref{LM:qsimplefrac}.
The proof of necessity follows the same general lines as the proof of necessity in Theorem~\ref{THM:ratsum}.
For the necessity, we assume that~$f$ is $(\tau_x, \tau_y)$-summable in~$\overline{\bF(x)}(y)$, i.e.,
there exist~$g, h\in \overline{\bF(x)}(y)$ such that
\begin{equation}\label{EQ:qthm}
f = \tau_x(g) - g + \tau_y(h) - h.
\end{equation}
We decompose the rational function~$g$ into the form
\begin{equation}\label{EQ:qthm2}
g = \tau_y(g_1)-g_1 + g_2 + \frac{\lambda_1}{(y- \mu_1)^j} + \cdots + \frac{\lambda_n}{(y-\mu_n)^j},
\end{equation}

where~$g_1, g_2\in \bF(x, y)$, $g_2$ is a rational function with poles having order
different from~$j$,~$\lambda_k, \mu_k\in \overline{\bF(x)}$, and the~$\mu_k$'s are in distinct~$q^\bZ$-orbits.

\smallskip
\noindent {\bf Claim 1.} For each~$i\in \{1, 2, \ldots, m\}$, at least one element of the set
\[\Lambda := \{\mu_1, \ldots, \mu_n, \tau_x(\mu_1), \ldots, \tau_x(\mu_n)\}\]
is in the same $q^\bZ$-orbit as~$\beta_i$. For each element~$\eta\in \Lambda$, there is   one
element of~$\Lambda \setminus \{\eta\} \cup \{\beta_1, \ldots, \beta_m\}$ that
is $q^\bZ$-equivalent to~$\eta$.

\smallskip
\noindent {\bf Proof of Claim 1.}
The argument is the same as the proof of Claim 1 in the proof of Theorem~\ref{THM:ratsum}, replacing~$\si_x$ with~$\tau_x$
and~$\text{dres}_y(f,\beta,j)$ with $\text{qres}_y(f,\beta,j).$

\smallskip
According to Claim 1, we have either~$\beta_i \simq \mu'_1$ or~$\beta_i \simq \tau_x(\mu'_1)$
for some~$\mu'_1$ in~$\{\mu_1, \ldots, \mu_n\}$. We shall deal with each case separately.
The proofs of the next two claims are essentially the same proofs of the corresponding
claims in the proof of Theorem~\ref{THM:ratsum} after one replaces~$\si_x$ with~$\tau_x$.

\smallskip
\noindent {\bf Claim 2.} Assume $\beta_i \simq \mu'_1$.\\[0.05in]
(a)  Fix  $\beta_i$ and $j \in \bN, j \geq 2$ and assume that $\tau_x^{\ell} \beta_i \not \sim_{q^\bZ}\beta_i$ for $1 \leq \ell \leq j-1$. Then there exist $\mu'_1, \ldots \mu'_j \in \{\mu_1, \ldots, \mu_n\}$ such that
\begin{equation}\label{EQ:claim2d}
\tau_x^{j-1}(\beta_i)\simq\mu'_j, \text{ and }
\end{equation}
\begin{equation}\label{EQ:claim2e}
\tau_x(\mu'_{1})\simq \mu'_{2}, \, \, \tau_x(\mu'_{2})\simq\mu'_{3},  \, \, \ldots,  \, \, \tau_x(\mu'_{j-1}) \simq\mu'_{{j}}.
\end{equation}
(b) There exists~$t_i\in \bN, t_i \neq 0$ such that~$t_i\leq n$ and~$\tau_x^{t_i}(\beta_i)/\beta_i\in q^\bZ$. For the smallest such $t_i$, there exist $\mu'_1, \ldots \mu'_{t_i-1} \in \{\mu_1, \ldots, \mu_n\}$ such that
\begin{equation}\label{EQ:claim2f}
\tau_x(\mu'_{1})\simq \mu'_{2}, \, \, \tau_x(\mu'_{2})\simq\mu'_{3},  \, \, \ldots,  \, \, \tau_x(\mu'_{t_i-1}) \simq \mu'_{{t_i}}, \tau_x(\mu'_{t_i}) \simq \beta_i.
\end{equation}

\smallskip
\noindent {\bf Claim 3.} Assume $\beta_i \simq \tau_x(\mu'_1)$.\\[0.05in]
(a)  Fix  $\beta_i$ and $j \in \bN, j \geq 2$ and assume that $\tau_x^{\ell} \beta_i \not\sim_{q^\bZ} \beta_i$ for $1 \leq \ell \leq j-1$. Then there exist $\mu'_1, \ldots \mu'_j \in \{\mu_1, \ldots, \mu_n\}$ such that
\begin{equation}\label{EQ:claim3d}
\beta_i\simq\tau_x^{j}(\mu'_j), \text{ and }
\end{equation}
\begin{equation}\label{EQ:claim3e}
\mu'_{1}\simq \tau_x(\mu'_{2}), \, \, \mu'_{2}\simq\tau_x( \mu'_{3}),  \, \, \ldots,  \, \, \mu'_{j-1}\simq \tau_x(\mu'_{{j}}).
\end{equation}
(b) There exists~$t_i\in \bN$ such that~$t_i\leq n$ and~$\tau_x^{t_i}(\beta_i)-\beta_i\in \bZ$. For the smallest such $t_i$, there exist $\mu'_1, \ldots \mu'_{t_i} \in \{\mu_1, \ldots, \mu_n\}$ such that
\begin{equation}\label{EQ:claim3f}
\mu'_{1}\simq \tau_x(\mu'_{2}), \, \, \mu'_{2}\simq \tau_x(\mu'_{3}),  \, \, \ldots,  \, \, \mu'_{t_i-1} \simq \tau_x(\mu'_{t_i}), \mu'_{t_i} \simq \beta_i.
\end{equation}

\noindent Using these claims, we now complete the proof.  Although the remainder of the proof is similar to the proof of
Theorem~\ref{THM:ratsum}, there are a few differences so we will give the details.\\[0.1in]
Fix some $\beta_i$ and assume, as in Claim 2, that $\beta_i \simq \mu'_1$, that is $\beta_i= q^{\omega_0}\mu'_1$.  From Claim 2(b) we have that
$\tau_x(\mu'_k) = q^{\omega_k} \mu'_{k+1}$ for~$k=1, \ldots, t_i-1$ and~$\tau_x(\mu'_{t_i})= q^{\omega_{t_i}}\beta_i$.
We can conclude that~$\tau_x^{t_i}(\beta_i)/\beta_i =q^{s_i} \in q^\bZ$, where~$s_i = \omega_0 + \omega_1 + \cdots + \omega_{t_i}$. This implies that~$\beta_i = c_i x^{{s_i}/{t_i}}$ for some~$c_i\in \bF$
by Lemma~\ref{LM:qintlin}.  \\[0.1in]
We now wish to compare the $q$-discrete residues at the $\beta_i$ on the left side of \eqref{EQ:qthm2} with the $q$-disrete residues at the elements of $\Lambda$ on the right side of \eqref{EQ:qthm2}. The equivalences of \eqref{EQ:claim2f} yield

\begin{center}
\begin{tabular}{|c|c|}  \hline
$q^\bZ$-orbit &  Comparison of two sides of \eqref{EQ:qthm2}\\ \hline
$\mu'_1, \tau_x(\mu'_{t_i})$ & $\alpha_{i,j} \ = \ q^{-j\omega_{t_i}}\tau_x(\lambda_{t_i}) - q^{j\omega_0}\lambda_1$ \\ \hline
$\mu'_{t_i}, \tau_x(\mu'_{t_i-1}) $ & $0  \ = \  q^{-j\omega_{t_i-1}}\tau_x(\lambda_{t_i-1}) - \lambda_{t_i}$\\ \hline
$\mu'_{t_i-1}, \tau_x(\mu'_{t_i-2}) $ & $0  \ = \  q^{-j\omega_{t_i-2}}\tau_x(\lambda_{t_i-2}) - \lambda_{t_i-1}$\\ \hline
  \vdots      &  \vdots     \\ \hline
    $\mu'_3, \tau_x(\mu'_{2})$ & $0        \ = \  q^{-j\omega_{2}}\tau_x(\lambda_{2}) - \lambda_{3}$\\ \hline
    $\mu'_2, \tau_x(\mu'_{1})$ & $0        \ = \ q^{-j\omega_{1}}\tau_x(\lambda_{1}) - \lambda_{2}$\\ \hline
\end{tabular}
\end{center}
Using the equations in the last column to eliminate all intermediate terms, one can show that $\alpha_{i,j} = q^{-js_i}\tau_x^{t_i}(\lambda) - \lambda$ where $\lambda = q^{jw_0}\lambda_1$.\\[0.2in]
We now turn to the situation of Claim 3, that is we assume that $\beta_i = q^{-\omega_0}\tau_x(\mu'_1)$. From Claim 3(b), we have that $\mu'_k = q^{-\omega_k}\tau_x(\mu'_{k+1})$ for $k = 1, \ldots , t_i-1$ and $\mu'_{t_i} = q^{-\omega_{t_i}}\beta_i$. We can conclude that~$\tau_x^{t_i}(\beta_i)/\beta_i =q^{s_i} \in q^\bZ$, where~$s_i = \omega_0 + \omega_1 + \cdots + \omega_{t_i}$. This implies that~$\beta_i = c_i x^{{s_i}/{t_i}}$ for some~$c_i\in \bF$
by Lemma~\ref{LM:qintlin}.  \\[0.1in]
We now wish to compare the $q$-discrete residues at the $\beta_i$ on the left side of \eqref{EQ:qthm2} with the $q$-disrete residues at the elements of $\Lambda$ on the right side of \eqref{EQ:qthm2}. The equivalences of \eqref{EQ:claim3f} yield

\begin{center}
\begin{tabular}{|c|c|}  \hline
$q^\bZ$-orbit &  Comparison of two sides of \eqref{EQ:qthm2}\\ \hline
$\mu'_{t_i}, \tau_x(\mu'_{1})$ & $\alpha_{i,j} \ = \ q^{-j\omega_{0}}\tau_x(\lambda_{1}) - q^{j\omega_{t_i}}\lambda_{t_i}$ \\ \hline
$\mu'_{1}, \tau_x(\mu'_{2}) $ & $0  \ = \  q^{-j\omega_{1}}\tau_x(\lambda_{2}) - \lambda_{1}$\\ \hline
$\mu'_{2}, \tau_x(\mu'_{3}) $ & $0  \ = \  q^{-j\omega_{2}}\tau_x(\lambda_{3}) - \lambda_{2}$\\ \hline
  \vdots      &  \vdots     \\ \hline
    $\mu'_{t_i-2}, \tau_x(\mu'_{t_i-1})$ & $0        \ = \  q^{-j\omega_{t_i-2}}\tau_x(\lambda_{t_i-1}) - \lambda_{t_i-2}$\\ \hline
    $\mu'_{t_i-1}, \tau_x(\mu'_{t_i})$ & $0        \ = \ q^{-j\omega_{t_i-1}}\tau_x(\lambda_{t_i}) - \lambda_{t_i-1}$\\ \hline
\end{tabular}
\end{center}

Using the equations in the last column to eliminate all intermediate terms, one can show
that $\alpha_{i,j} = q^{-js_i}\tau_x^{t_i}(\lambda)-\lambda$ where $\lambda = q^{j\omega_{t_i}}\lambda_{t_i}$.
Since~$\lambda_{t_i}\in \bF(x)(\beta_i)$ and~$\beta_i\in \bF(x^{1/t_i})$, we have~$\lambda_{t_i} \in \bF(x^{1/t_i})$.
This completes the proof. \QED
\end{proof}

\begin{remark}\label{RM:t}
Let~$\alpha\in \bF(x^{1/t})$ with~$t\in \bN\setminus \{0\}$ and~$m\in \bZ$. We show how to reduce the problem
of deciding whether there exists~$\beta\in \bF(x^{1/t})$ such that
\begin{equation}\label{EQ:t}
  \alpha = q^m \tau_x^t(\beta) -\beta,
\end{equation}
to the usual $q$-summability problem as in Section~\ref{SUBSECT:uqshift}.
First, set~$\bar x = x^{1/t}$ and~$\bar \tau_{\bar x} = \tau_x^t$. Then~$\bar \tau_{\bar x}(\bar x) = q\bar x$, $\alpha \in \bF(\bar x)$ and~\eqref{EQ:t} is
equivalent to
\begin{equation*}
  \alpha = q^m \bar \tau_{\bar x}(\beta) - \beta \quad \text{for some~$\beta \in \bF(\bar x)$.}
\end{equation*}
Let~$\bar \beta = {\bar x}^m \beta$ and~$\bar \alpha= \bar{x}^m \alpha$. By a direct calculation, we
have that $q^m \bar \tau_{\bar x}(\beta) = q^m \tau_x({\bar \beta}/{\bar{x}^m})$ $  = \tau_{\bar x}(\bar \beta)/\bar{x}^m$,
which implies that
\[\bar{\alpha} = \tau_{\bar x}(\bar \beta) - \bar{\beta}. \]
Therefore, we can use the criterion on the $q$-summability in~$\bF(\bar x)$ in Section~\ref{SUBSECT:uqshift} to solve this problem.
\end{remark}

\begin{example}\label{EXAM:qratsumnru}
Let~$f = 1/(x^n+y^n)$ with~$n\in \bN\setminus \{0\}$. Over the field~$\overline{\bF(x)}$, we can decompose~$f$ into
the form~\eqref{EQ:exam}. By Theorem~\ref{THM:qratsum}, $f$ is~$(\tau_x, \tau_y)$-summable in~$\bF(x, y)$ if and only if
for all~$i\in \{1, \ldots, n\}$,
\begin{equation}\label{EQ:examq}
\frac{1}{n(\omega_i x)^{n-1}} = q^{-1} \tau_x(\gamma_i)- \gamma_i \quad \text{for
some~$\gamma_i\in \bF(x)$.}
\end{equation}
By Remark~\ref{RM:t}, the equation~\eqref{EQ:examq} is equivalent to
\[\frac{1}{n\omega_i^{n-1} x^{n}} = \tau_x(\bar \gamma_i)- \bar \gamma_i \quad \text{for
some~$\bar \gamma_i\in \bF(x)$.}\]
By the $q$-summability criterion in Section~\ref{SUBSECT:uqshift}, we have~$\frac{1}{n\omega_i^{n-1} x^{n}}$ is $\tau_x$-summable in~$\bF(x)$
for all~$i\in \{1, \ldots, n\}$. Therefore, $f$ is~$(\tau_x, \tau_y)$-summable in~$\bF(x, y)$. In fact, we have
\begin{equation}\label{EQ:examq2}
\frac{1}{x^n+y^n} = \tau_x\left(\frac{c_n}{x^n + y^n}\right)- \frac{c_n}{x^n + y^n} +
\tau_y\left(\frac{c_n}{q^nx^n + y^n}\right)- \frac{c_n}{q^nx^n + y^n},
\end{equation}
where~$c_n = q^n/(1-q^n)$. In order to translate the identity~\eqref{EQ:examq2} into that of usual sums, we define
the transformation~$\rho: \bF(x, y) \rightarrow \bF(q^a, q^b)$ by~$\rho(x)=q^a$,~$\rho(y)=q^b$ and~$\rho(c)=c$ for any~$c\in \bF$.
 Since~$q$ is not a root of unity, $\rho$ is an isomorphism between two fields~$\bF(x, y)$ and~$\bF(q^a, q^b)$.
Let~$\si_a$ and~$\si_b$ denote the shift operators with respect to~$a$ and~$b$, respectively.
Then~$\rho (\tau_x(f)) = \si_a (\rho(f))$ and~$\rho (\tau_y(f)) = \si_b (\tau_y(f))$ for all~$f\in \bF(x, y)$.
Assume that~$\bF = \bC$ and~$\abs{q} >1$. Now the identity~\eqref{EQ:examq2} leads to the identity
\begin{align*}
  \sum_{a\geq 1}\sum_{b\geq 1}\frac{1}{q^{an} + q^{bn}} & = \frac{q^n}{1-q^n}\left(\sum_{b\geq 1}\frac{-1}{q^n+q^{bn}}
+ \sum_{a\geq 1}\frac{-1}{q^n+q^{(a+1)n}} \right) \\
   & =\frac{q^n}{1-q^n}\left(\frac{-1}{q^n+q^{n}}
+ 2 \sum_{a\geq 1}\frac{-1}{q^n+q^{(a+1)n}} \right)\\
& = \frac{1}{1-q^n}\left(-\frac{1}{2} + 2{\rm L}_1(-1;\frac{1}{q^n})\right),\\
 & \\
 \text{where} \  {\rm L}_1(x;q) & = x \sum_{a=1}^\infty \frac{1}{p^a - x}, \ \ \ |x| < |p|, p = q^{-1}.
\end{align*}
The function ${\rm L}_1(x;q)$ is called the $q$-logarithm (see~\cite{Zudilin2005} and the references in this paper).
 Moreover, P.\ Borwein has proved in~\cite{Borwein1991} that~${\rm L}_1(-1;\frac{1}{q^n})$ is irrational, which implies
that the double sum above is also irrational.

In this way, we reduce double sums into single ones and then evaluate these in terms of values of special functions.
\end{example}

\begin{example}
A $q$-analogue of Tornheim's double sums is presented by Zhou et al. in \cite{Zhou2008}, which is of the form
\[T[r, s, t; \si, \tau] = \sum_{n, m =1}^{\infty} \frac{\si^n\tau^mq^{(r+t-1)n+(s+t-1)m}}{[n]_q^r [m]_q^s [n+m]_q^t},\]
where~$\si, \tau \in \{-1, 1\}$ and~$[n]_q := \sum_{j=0}^{n-1} = \frac{q^n-1}{q-1}$.
We consider the special case when~$\si=\tau=1$ and~$r=s=0$. By setting~$x=q^n$ and~$y= q^m$, the summand of~$T[0, 0, t; 1, 1]$ is
the rational function
\[f = \frac{(xy)^{t-1}(q-1)^t}{(xy-1)^t},\quad \text{where~$t\in \bN\setminus\{0\}$}.\]
We show that~$f$ is not~$(\tau_x, \tau_y)$-summable in~$\bF(x, y)$ for all~$t\in \bN\setminus\{0\}$.
The partial fraction decomposition of~$f$ with respect to~$y$ is
\[f = \sum_{i=0}^{t-1} \frac{\alpha_i}{(y-1/x)^{t-i}}, \quad \text{where~$\alpha_i = (q-1)^t \binom{t-1}{i}x^{-(t-i)}$}.\]
By Theorem~\ref{THM:qratsum}, $f$ is~$(\tau_x, \tau_y)$-summable in~$\bF(x, y)$ if and only if for all~$i\in \{0, 1, \ldots, t-1\}$, we have
\begin{equation}\label{EQ:exam3}
\alpha_i = q^{t-i}\tau_x(\gamma_i) - \gamma_i \quad \text{for some~$\gamma_i \in \bF(x)$.}
\end{equation}
By Remark~\ref{RM:t}, \eqref{EQ:exam3} is equivalent to
\[(q-1)^t \binom{t-1}{i} = \tau_x(\bar \gamma_i) - \bar \gamma_i \quad \text{for some~$\bar \gamma_i \in \bF(x)$.} \]
By the criterion for the~$q$-summability in Section~\ref{SUBSECT:uqshift}, the nonzero
constant \linebreak $(q-1)^t \binom{t-1}{i}\in \bF$ is not~$\tau_x$-summable in~$\bF(x)$, which implies
that~$f$ is not $(\tau_x, \tau_y)$-summable in~$\bF(x, y)$.
\end{example}

\bibliographystyle{plain}

%
%

\end{document}